\documentclass[oneside]{amsart}
\usepackage[utf8]{inputenc}
\usepackage{amsmath,amstext,amsopn,amssymb, amsfonts, amsthm, mathtools}
\usepackage[dvipsnames]{xcolor}
\usepackage{geometry}
\usepackage{bbm}
\usepackage{pdfsync}
\usepackage{tikz}
\usepackage{enumitem}
\usepackage{subfig}
\usepackage{caption}
\usepackage{graphicx}
\usepackage{mathrsfs}
\usepackage[ruled,vlined,algosection]{algorithm2e}
\DeclareMathAlphabet{\mathpzc}{OT1}{pzc}{m}{it}

\usepackage[pdftex,bookmarks,colorlinks,breaklinks]{hyperref}
\definecolor{dullmagenta}{rgb}{0.4,0,0.4}   
\definecolor{darkblue}{rgb}{0,0,0.4}
\definecolor{darkgreen}{rgb}{0,0.4,0}

\hypersetup{
  linkcolor=darkblue,
  citecolor=blue,
  filecolor=dullmagenta,
  urlcolor=darkblue
}

\newtheorem{TheoremLetter}{Theorem}
{}

\newtheorem{definition}{Definition}
\newtheorem*{definition*}{Definition}

\newtheorem{theorem}{Theorem}
\newtheorem*{theorem*}{Theorem}

\newtheorem*{conjecture*}{Conjecture}

\newtheorem*{question*}{Question}

\newtheorem{lemma}[theorem]{Lemma}
\newtheorem*{lemma*}{Lemma}

\newtheorem{proposition}[theorem]{Proposition}

\newtheorem*{corollary*}{Corollary}

\theoremstyle{definition}

\newtheorem*{remark*}{Remark}

\theoremstyle{plain}

\newtheorem*{example*}{Example}

\numberwithin{equation}{section}
\numberwithin{theorem}{section}

\makeatletter
\newcommand{\customlabel}[2]{
   \protected@write \@auxout {}{
     \string \newlabel {#1}{{#2}{\thepage}{#2}{#1}{}} }
   \hypertarget{#1}{#2}
}
\makeatother

\def\XXint#1#2#3{{\setbox0=\hbox{$#1{#2#3}{\int}$}
     \vcenter{\hbox{$#2#3$}}\kern-.5\wd0}}

\newcommand{\meetconst}{\operatorname{Meet}}

\makeindex
\begin{document}

\title{Pairwise meets of antichains in $\mathbb{Z}^d$}
\address{Universidad Aut\'onoma de Madrid}
\author{Guillermo Rey}
\email{guillermo.rey@uam.es}
\thanks{Research supported in part by grants PID2022-139521NA-I00 and
  RYC2024-051323-I, both funded by MICIU/AEI/10.13039/501100011033.}
\maketitle

\begin{abstract}
  The meet $a \land b$ of two points $a, b \in \mathbb{Z}^d$ is their
  coordinatewise minimum. We show that every finite antichain
  $\mathcal{A} \subseteq \mathbb{Z}^d$ has at least
  $c_d |\mathcal{A}|^{\frac{d}{d-1}}$ distinct pairwise meets, where
  $c_d > 0$ depends only on $d$, and that the exponent $\frac{d}{d-1}$ is best
  possible. As a corollary we obtain an isoperimetric inequality for downsets:
  every finite downset $\mathcal{D} \subseteq \mathbb{Z}^d_{\geq 0}$ satisfies
  $|\mathcal{D}| \geq c_d |\max(\mathcal{D})|^{\frac{d}{d-1}}$, where
  $\max(\mathcal{D})$ is its set of maximal elements. By prime factorization
  the meet bound also yields a lower bound on greatest common divisors: a
  primitive set of $N$ integers supported on at most $d$ primes has at least
  $c_d N^{\frac{d}{d-1}}$ distinct pairwise gcds.
\end{abstract}

\section{Introduction}

The meet of two points in $\mathbb{Z}^d$ is their coordinatewise minimum:
\begin{align*}
  (a \land b)_i = \min(a_i, b_i).
\end{align*}
Throughout, $\mathbb{Z}^d$ carries the \emph{product order}, in which $a \leq b$
means $a_i \leq b_i$ for every $i$.
For a set $\mathcal{A} \subseteq \mathbb{Z}^d$ we write $\mathcal{A} \land
\mathcal{A}$ for its set of pairwise meets:
$\mathcal{A} \land \mathcal{A} := \{a \land b:\, a,b \in \mathcal{A}\}$.
The question we are interested in is: how small can $\mathcal{A} \land
\mathcal{A}$ be in terms of the size of $\mathcal{A}$?

If $\mathcal{A}$ is a \emph{chain}, that is, any two of its elements are
comparable, then the meet of two points is simply the smaller of them and
$\mathcal{A} \land \mathcal{A} = \mathcal{A}$. This makes the problem trivial,
since the meet of two elements of a chain is again one of them. We therefore
turn to the opposite situation, where comparability is forbidden: if
$\mathcal{A}$ is an \emph{antichain} (a set no two distinct elements of which
are comparable), how small can $\mathcal{A} \land \mathcal{A}$ be?

Antichains can nonetheless be very thin, lying for instance on a single
hyperplane, so one might expect them to produce few pairwise meets. Our main
theorem shows that this cannot happen: the meet set always grows like a fixed
power of the size of the antichain.

\begin{TheoremLetter} \label{TheoremA}
  For every $d \geq 2$ there exists a constant $c_d > 0$ such that for every
  finite antichain $\mathcal{A} \subseteq \mathbb{Z}^d$
  \begin{align*}
    |\mathcal{A} \land \mathcal{A}| \geq c_d |\mathcal{A}|^{\frac{d}{d-1}}.
  \end{align*}
\end{TheoremLetter}

The exponent $\frac{d}{d-1}$ is best possible, as shown by the standard simplex.
\begin{example*}
  Fix a dimension $d \geq 2$ and an integer $L \geq 1$, and let
  \begin{align*}
    \mathcal{A} = \{m \in \mathbb{Z}^d_{\geq 0} :\, |m| = L\},
  \end{align*}
  where $|m| = \sum_{i=1}^d m_i$.
  Then $\mathcal{A}$ is an antichain and
  $\mathcal{A} \land \mathcal{A} = \{m \in \mathbb{Z}^d_{\geq 0}:\, |m| \leq L\}$.
  In particular,
  \begin{align*}
    |\mathcal{A} \land \mathcal{A}| \asymp |\mathcal{A}|^{\frac{d}{d-1}},
  \end{align*}
  where $\asymp$ denotes equality up to constant factors depending only on $d$.
\end{example*}

Theorem \ref{TheoremA} has a direct corollary, which can be read as an
isoperimetric inequality for \emph{downsets} in $\mathbb{Z}^d_{\geq 0}$, sets
$\mathcal{D}$ such that $0 \leq x \leq y \in \mathcal{D}$ implies
$x \in \mathcal{D}$.
\begin{TheoremLetter}
  Let $\mathcal{D} \subseteq \mathbb{Z}^d_{\geq 0}$ be a finite downset, and let
  $\max(\mathcal{D})$ be its set of maximal elements. Then
  \begin{align*}
    |\mathcal{D}| \geq c_d |\max(\mathcal{D})|^{\frac{d}{d-1}}.
  \end{align*}
\end{TheoremLetter}
Indeed, $\max(\mathcal{D})$ is an antichain, and since $\mathcal{D}$ is a
downset we have $\max(\mathcal{D}) \land \max(\mathcal{D}) \subseteq \mathcal{D}$
(if $a, b \in \mathcal{D}$ then $a \land b \leq a$, so $a \land b \in
\mathcal{D}$); the bound then follows by applying Theorem \ref{TheoremA} to the
antichain $\max(\mathcal{D})$.
The downset bound is reminiscent of shadow inequalities in products of chains,
such as the theorem of Clements and Lindström \cite{clements_lindstrom}
extending that of Kruskal and Katona \cite{kruskal, katona}; we do not pursue it
further, since Theorem \ref{TheoremA} is stronger, $\mathcal{A} \land
\mathcal{A}$ being in general a sparse subset of the downset generated by
$\mathcal{A}$.

Theorem \ref{TheoremA} also has an application in number theory.
Let $P = \{p_1, \dots, p_d\}$ be a set of prime numbers, and let $m$ and $n$ be
integers whose prime factorizations involve only primes in $P$, say
\begin{align*}
  m = \prod_{i=1}^d p_i^{a_i} \quad \text{and} \quad n = \prod_{i=1}^d
  p_i^{b_i}, \qquad a, b \in \mathbb{Z}^d_{\geq 0}.
\end{align*}
Then $\gcd(m,n)$ is the integer whose exponent vector is the meet $a \land b$:
\begin{align*}
  \gcd(m,n) = \prod_{i=1}^d p_i^{(a \land b)_i},
\end{align*}
and $m \mid n$ if and only if $a \leq b$.
One can therefore ask, for a set $A \subseteq \mathbb{N}$, how many distinct
greatest common divisors arise from pairs of its elements. Write
\begin{align*}
  \operatorname{GCD}(A) = \{\gcd(m,n):\, m,n \in A\}.
\end{align*}

As before, $|\operatorname{GCD}(A)|$ admits no good lower bound without further
hypotheses on $A$. The natural restriction is to assume that $A$ is a
\emph{primitive set}, that is, one in which no element divides another. Even so,
a set of $N$ distinct primes is primitive and has
$|\operatorname{GCD}(A)| = N+1$: the only gcds are the primes themselves
and $1$.

We say that a primitive set $A$ is \emph{supported} on a set of primes $P$ if
every element of $A$ has a prime factorization involving only primes in $P$. It
is then natural to ask, as $|A| \to \infty$, how small $|\operatorname{GCD}(A)|$
can be for primitive sets supported on $d$ primes.

Theorem \ref{TheoremA} has the following direct corollary.
\begin{TheoremLetter}
  For every $d \geq 2$ there exists a constant $c_d > 0$ such that every
  primitive set $A$ supported on at most $d$ primes has
  \begin{align*}
    |\operatorname{GCD}(A)| \geq c_d |A|^{\frac{d}{d-1}}.
  \end{align*}
\end{TheoremLetter}
Indeed, the exponent-vector map sends $A$ bijectively to a finite antichain
$\mathcal{A} \subseteq \mathbb{Z}^d_{\geq 0}$ under which $\gcd$ becomes $\land$,
so $|\operatorname{GCD}(A)| = |\mathcal{A} \land \mathcal{A}|$ and
Theorem \ref{TheoremA} applies.
The same result holds with $\gcd$ replaced by $\operatorname{lcm}$: the map
$\mathcal{A} \mapsto -\mathcal{A}$ sends antichains to antichains and satisfies
$(-a) \land (-b) = -(a \lor b)$, so applying Theorem \ref{TheoremA} to
$-\mathcal{A}$ bounds the number of distinct pairwise joins, hence of distinct
lcms.

Sets of the form $\mathcal{A} \land \mathcal{A}$, and more generally
$\mathcal{A} \land \mathcal{B}$, have been studied
before. Daykin, Kleitman and West determined in \cite{dkw} the minimum of
$|\mathcal{A} \land \mathcal{B}|$ over subsets of prescribed sizes in a product
of chains (such as $\mathbb{Z}^d$); there the minimizers are order ideals (downsets),
which are as far from antichains as a set can be, and it is precisely the
antichain hypothesis that excludes them and forces the power $\frac{d}{d-1}$.
The exponent, and the shape of our argument, are moreover close to the
antichain projection inequalities of Engel, Mitsis, Pelekis and Reiher
\cite{emtp} and of Janzer \cite{janzerb}, who compare an antichain to its
lower-dimensional coordinate projections. Their setting is in fact
complementary to ours. They work with \emph{weak} antichains, that is, sets
containing no two points $x, y$ with $x_i < y_i$ in every coordinate, and
study how the size of such a set compares with those of its codimension-one
projections $\pi_i^*$. For a \emph{strong} antichain these projections carry no
information: each $\pi_i^*$ is injective (Lemma \ref{injective}), so
$|\pi_i^*(\mathcal{A})| = |\mathcal{A}|$ for every $i$, and the expansion we
exploit is invisible to projection sizes alone. The two problems thus sit on
opposite sides of the same landscape: projection inequalities are interesting
precisely when comparability in individual coordinates is allowed, whereas the
meet set is interesting precisely when it is not. As in their work we induct on the
dimension, slice $\mathcal{A}$ along the fibers of a projection, bound each
fiber by the
inductive hypothesis, and reassemble the estimates with H\"older's
inequality; a technique originating with Loomis and Whitney \cite{lw} and with
Bollob\'as and Thomason \cite{bollobas_thomason}.
We note that the extremal antichain here is the simplex, whereas the gap in
\cite{emtp, janzerb} is minimized by the staircase set
$\{x \in \{0, \dots, L-1\}^d : x_i = 0 \text{ for some } i\}$; the difference
reflects the same shift between the two regimes.

We prove Theorem \ref{TheoremA} by induction on the dimension, the case $d=2$
being an exact count. The inductive step combines two ingredients: a slicing
argument that converts a small coordinate projection of $\mathcal{A}$ into many
meets, and a pigeonhole argument that guarantees a positive fraction of pairs
whose meets are pinned down along a fixed coordinate split. The pigeonhole organizes
pairs by their comparison-type, the set of coordinates in which one point lies
below the other. A simple injectivity property of codimension-one projections
is used repeatedly.

\section{Proof of Theorem \ref{TheoremA}}

For a finite set $S$ we write $|S|$ for its cardinality, and we set
$[d] = \{1, \dots, d\}$.
We call $\mathcal{I} \subseteq [d]$ a \emph{proper} subset if
$\mathcal{I} \neq \emptyset$ and $\mathcal{I} \neq [d]$.
Writing $k = |\mathcal{I}|$, we define two projections
\begin{align*}
  \pi_{\mathcal{I}} &: \mathbb{Z}^d \to \mathbb{Z}^k \\
  \text{and }\pi_{\mathcal{I}}^* &: \mathbb{Z}^d \to \mathbb{Z}^{d-k},
\end{align*}
where $\pi_{\mathcal{I}}$ is the projection that only keeps the coordinates with
indices in $\mathcal{I}$, and $\pi_{\mathcal{I}}^*$ keeps those in the
complement of $\mathcal{I}$. In particular, if
$\mathcal{I} = \{i_1, \dots, i_k\}$ with $i_1 < i_2 < \dots < i_k$ then
\begin{align*}
  \pi_{\mathcal{I}}(x_1, \dots, x_d) = (x_{i_1}, x_{i_2}, \dots, x_{i_k}).
\end{align*}
When $\mathcal{I}$ consists of a single index $i$, we abbreviate
$\pi_{\{i\}}$ by $\pi_i$.

We begin with a simple lemma which will be useful in the proof.
\begin{lemma} \label{injective}
  Let $\mathcal{A} \subseteq \mathbb{Z}^d$ be an antichain. Then for every
  $1 \leq i \leq d$ the projection $\pi_{i}^*$ is injective.
\end{lemma}
\begin{proof}
  Suppose there are two elements $a,b \in \mathcal{A}$ with
  $\pi_i^*(a) = \pi_i^*(b)$.
  Their $i$-th coordinates must be ordered, so we assume without loss of
  generality that $a_i \leq b_i$. But then we have $a \leq b$ since all the
  other coordinates are equal, which implies $a = b$ since $\mathcal{A}$ is an
  antichain.
\end{proof}

\begin{definition}
  For a fixed dimension $d \geq 2$, define $\meetconst(d)$ as the largest
  constant $c$ such that for every finite antichain
  $\mathcal{A} \subseteq \mathbb{Z}^d$
  \begin{align*}
    |\mathcal{A} \land \mathcal{A}| \geq c|\mathcal{A}|^{\frac{d}{d-1}}.
  \end{align*}
\end{definition}
In two dimensions the situation is easier and we can prove
\begin{proposition} \label{dim2_prop}
  Let $\mathcal{A} \subseteq \mathbb{Z}^2$ be an antichain with $N$ elements.
  Then
  \begin{align} \label{dim2_easy}
    |\mathcal{A} \land \mathcal{A}| = \binom{N+1}{2}.
  \end{align}
\end{proposition}
In particular, $\meetconst(2) = \frac{1}{2}$.
\begin{proof}
  To see \eqref{dim2_easy}, enumerate the points in
  $\mathcal{A}$ by $\{p_1, \dots, p_N\}$.
  Let $p_i = (x_i,y_i)$, and assume without loss of generality that $\{x_i\}$ is
  monotone increasing. By Lemma \ref{injective} the sequence $\{x_i\}$ is
  strictly increasing and, using the antichain condition, $\{y_i\}$ is strictly
  decreasing.

  Now, observe that if $i \leq j$ then $p_i \land p_j = (x_i, y_j)$, therefore
  \begin{align*} \mathcal{A} \land \mathcal{A} = \{
      & (x_1, y_1), (x_1, y_2), (x_1, y_3), \ldots, (x_1, y_N), \\
      & \phantom{(x_1, y_1),{}} (x_2, y_2), (x_2, y_3), \ldots, (x_2, y_N), \\
      & \phantom{(x_1, y_1), (x_1, y_2),{}} (x_3, y_3), \ldots, (x_3, y_N), \\
      & \phantom{(x_1, y_1), (x_1, y_2), (x_1, y_3),{}} \ddots \\
      & \phantom{(x_1, y_1), (x_1, y_2), (x_1, y_3)\ldots,{}} (x_N, y_N) \} \\
    &= \{(x_i, y_j): 1 \leq i \leq j \leq N\}.
  \end{align*}
  This implies
  \begin{align*}
    |\mathcal{A} \land \mathcal{A}| = \sum_{m=1}^N m = \binom{N+1}{2}.
  \end{align*}
\end{proof}

The next proposition allows us to guarantee many different meets whenever
the projections along some index set are sufficiently ``compressed''.
This will be one of the main inductive tools in the argument.
\begin{proposition} \label{induction}
  Let $\mathcal{A}$ be a finite antichain in $\mathbb{Z}^d$ with $d \geq 3$.
  Suppose there is an index set $\mathcal{I} \subseteq [d]$ with
  $1 \leq |\mathcal{I}| \leq d-2$ and a constant $C$ such that
  $|\pi_{\mathcal{I}}(\mathcal{A})| \leq C|\mathcal{A}|^{\frac{k}{d-1}}$, where
  $k = |\mathcal{I}|$. Then
  \begin{align*}
    |\mathcal{A} \land \mathcal{A}| \geq
    C^{\frac{-1}{d-k-1}} \meetconst(d-k) |\mathcal{A}|^{\frac{d}{d-1}}.
  \end{align*}
\end{proposition}
\begin{proof}
  For brevity we let $\pi = \pi_{\mathcal{I}}$ and
  $\pi^* = \pi_{\mathcal{I}}^*$.
  We will also abbreviate the meet set $\mathcal{E} \land \mathcal{E}$ by
  $\mathcal{E}^{\land}$, for any $\mathcal{E} \subseteq \mathbb{Z}^d$.
  Let $\mathcal{P} = \pi(\mathcal{A})$, and for every $u \in \mathcal{P}$
  we define the section
  \begin{align*}
    \mathcal{A}_u = \{a \in \mathcal{A}:\, \pi(a) = u\}.
  \end{align*}
  First, observe that $\pi^*(\mathcal{A}_u)$ is an antichain in
  $\mathbb{Z}^{d-k}$.
  Indeed, let $a,b \in \mathcal{A}_u$. If $\pi^*(a) \leq \pi^*(b)$ then
  $a \leq b$ since $\pi(a) = \pi(b) = u$, so $a = b$ and thus
  $\pi^*(a) = \pi^*(b)$.
  Moreover, since $\pi$ is constant on $\mathcal{A}_u$, the map $\pi^*$ restricts
  to a bijection $\mathcal{A}_u \to \pi^*(\mathcal{A}_u)$ that carries
  $\mathcal{A}_u^{\land}$ onto $(\pi^*(\mathcal{A}_u))^{\land}$; in particular
  $|\mathcal{A}_u^{\land}| = |(\pi^*(\mathcal{A}_u))^{\land}|$ and
  $|\mathcal{A}_u| = |\pi^*(\mathcal{A}_u)|$.

  If $c \in \mathcal{A}_u^{\land}$ then $\pi(c) = u$, therefore for distinct
  $u,v \in \mathcal{P}$, we have
  $\mathcal{A}_u^{\land} \cap \mathcal{A}_v^{\land} = \emptyset$.

  These facts combined mean that
  \begin{align*}
    |\mathcal{A}^{\land}| &\geq
    \sum_{u \in \mathcal{P}} |\mathcal{A}_u^{\land}| \\
    &\geq \meetconst(d-k) \sum_{u \in \mathcal{P}}
    |\mathcal{A}_u|^{\frac{d-k}{d-k-1}}.
  \end{align*}

  Let $N = |\mathcal{A}|$. Writing $(d-k)' = (d-k)/(d-k-1)$ for the conjugate
  exponent of $d-k$, H\"older's inequality gives
  \begin{align*}
    N = \sum_{u \in \mathcal{P}} |\mathcal{A}_u| \leq
    |\mathcal{P}|^{\frac{1}{d-k}} \Big(
    \sum_{u \in \mathcal{P}} |\mathcal{A}_u|^{(d-k)'}
    \Big)^{\frac{1}{(d-k)'}}.
  \end{align*}
  Therefore, reorganizing we arrive at
  \begin{align*}
    \sum_{u \in \mathcal{P}} |\mathcal{A}_u|^{\frac{d-k}{d-k-1}} &\geq \Big(
    N |\mathcal{P}|^{-\frac{1}{d-k}} \Big)^{(d-k)'} \\
    &\geq C^{-\frac{(d-k)'}{d-k}} N^{(d-k)'\big(1-\frac{k}{(d-1)(d-k)}\big)} \\
    &= C^{-\frac{1}{d-k-1}} N^{\frac{d}{d-1}}.
  \end{align*}
\end{proof}

The next proposition provides a large collection of pairs in $\mathcal{A} \times \mathcal{A}$ for which their meets have a fixed form.
\begin{proposition} \label{type_pigeonhole}
  Let $\mathcal{A}$ be an antichain in $\mathbb{Z}^d$ with $d \geq 3$ and $N$
  elements.
  Then there exists an index set $\mathcal{I} \subseteq [d]$ with
  $1 \leq |\mathcal{I}| \leq d-2$ and a set $\mathcal{F} \subseteq
  \mathcal{A} \times \mathcal{A}$ such that
  \begin{align*}
    |\mathcal{F}| \geq \frac{N^2}{2^{d+1}}
  \end{align*}
  and such that every pair $(a,b) \in \mathcal{F}$ satisfies
  \begin{align*}
    \pi_{\mathcal{I}}(a) \leq \pi_{\mathcal{I}}(b) \quad\text{and}\quad
    \pi_{\mathcal{I}}^*(a) \geq \pi_{\mathcal{I}}^*(b).
  \end{align*}
\end{proposition}
\begin{proof}
  We can assume without loss of generality that $N \geq 2$.
  Define the type function by
  \begin{align*}
    T(a,b) = \{i \in [d]:\, a_i \leq b_i\},
  \end{align*}
  defined only for $a \neq b$.
  Since $\mathcal{A}$ is an antichain, $T(a,b)$ is always a proper non-empty
  subset of $[d]$. Thus, there are at most $2^d - 2$ possible different types.

  There are $N(N-1)$ pairs $(a,b) \in \mathcal{A} \times \mathcal{A}$ with
  $a \neq b$, so by the pigeonhole principle there must exist a proper index set
  $\mathcal{T} \subseteq [d]$ and a collection
  \begin{align*}
    \mathcal{P} = \{(a,b) \in \mathcal{A} \times \mathcal{A}:\, a \neq b \text{
      and } T(a,b) = \mathcal{T}\}
  \end{align*}
  such that
  \begin{align*}
    |\mathcal{P}| \geq \frac{N(N-1)}{2^{d} -2} \geq \frac{N^2}{2^{d+1}}.
  \end{align*}

  If $1 \leq |\mathcal{T}| \leq d-2$, then we are done by setting
  $\mathcal{I} = \mathcal{T}$ and $\mathcal{F} = \mathcal{P}$: for
  $(a,b) \in \mathcal{P}$ we have $a_i \leq b_i$ exactly for $i \in \mathcal{T}$,
  so $\pi_{\mathcal{I}}(a) \leq \pi_{\mathcal{I}}(b)$, while $a_i > b_i$ for
  $i \notin \mathcal{I}$, giving $\pi_{\mathcal{I}}^*(a) \geq \pi_{\mathcal{I}}^*(b)$.
  However, it could happen that $|\mathcal{T}| = d-1$. In this case we instead
  set $\mathcal{I} := [d] \setminus \mathcal{T}$ and transpose $\mathcal{P}$,
  \begin{align*}
    \mathcal{F} := \{(b,a):\, (a,b) \in \mathcal{P}\}.
  \end{align*}
  Here $|\mathcal{I}| = 1 \leq d-2$ since $d \geq 3$; for $(a,b) \in \mathcal{P}$
  the single index $i \in \mathcal{I}$ has $a_i > b_i$, while every
  $j \in \mathcal{T}$ has $a_j \leq b_j$, so the swapped pair
  $(b,a) \in \mathcal{F}$ satisfies
  $\pi_{\mathcal{I}}(b) \leq \pi_{\mathcal{I}}(a)$ and
  $\pi_{\mathcal{I}}^*(b) \geq \pi_{\mathcal{I}}^*(a)$.
\end{proof}

The proposition above lets us work with many pairs $(a,b)$ whose meet has a
particularly nice form. Indeed, for the set of pairs $\mathcal{F}$ constructed
above, the meet $a \land b$ is the unique point in $\mathbb{Z}^d$ such that
\begin{align*}
  \pi_{\mathcal{I}}(a \land b) = \pi_{\mathcal{I}}(a) \quad \text{and} \quad
  \pi_{\mathcal{I}}^*(a \land b) = \pi_{\mathcal{I}}^*(b).
\end{align*}
On the coordinates in $\mathcal{I}$ we have
$\pi_{\mathcal{I}}(a) \leq \pi_{\mathcal{I}}(b)$, so the coordinatewise minimum
agrees with $a$; on the remaining coordinates
$\pi_{\mathcal{I}}^*(a) \geq \pi_{\mathcal{I}}^*(b)$, so it agrees with $b$. As
$\mathcal{I}$ and its complement partition the coordinates, these two
projections determine the point.

The next theorem is the main result of this paper. In it, we will track the
constants rather precisely, but we make no claim that these are optimal.
\begin{theorem} \label{theorem_constants}
  For dimensions $d \geq 2$ we have
  \begin{align*}
    \meetconst(d) \geq 2^{-\gamma_d}
  \end{align*}
  where
  \begin{align*}
    \gamma_d = \frac{3d^2 + 15d - 40}{2}.
  \end{align*}
\end{theorem}
\begin{proof}
  Recall that in Proposition \ref{dim2_prop} we showed that $\meetconst(2) =
  \frac{1}{2}$. So the $d=2$ case follows by observing that $\gamma_2 = 1$.
  Now fix $d \geq 3$ and assume by induction that $\meetconst(n) \geq
  2^{-\gamma_n}$ for all $2 \leq n < d$.

  Let $\mathcal{A} \subseteq \mathbb{Z}^d$ be an antichain, and set $N =
  |\mathcal{A}|$. We can assume without loss of generality that $N \geq 2$.
  By Proposition \ref{type_pigeonhole} there exists an index set $\mathcal{I}$
  with $1 \leq k \leq d-2$ indices, where $k = |\mathcal{I}|$, and a
  collection $\mathcal{F}$ of pairs
  in $\mathcal{A} \times \mathcal{A}$ such that
  \begin{align*}
    |\mathcal{F}| \geq \frac{N^2}{2^{d+1}},
  \end{align*}
  and such that, for every pair $(a,b) \in \mathcal{F}$, their meet $a \land b$
  is the unique point satisfying
  \begin{align*}
    \pi_{\mathcal{I}}(a \land b) = \pi_{\mathcal{I}}(a) \quad \text{and} \quad
    \pi_{\mathcal{I}}^*(a \land b) = \pi_{\mathcal{I}}^*(b).
  \end{align*}
  Thus, pairs $(a,b), (a',b') \in \mathcal{F}$ will produce different meets
  whenever the corresponding projections differ:
  \begin{align*}
    \pi_{\mathcal{I}}(a) \neq \pi_{\mathcal{I}}(a') \quad\text{or}\quad
    \pi_{\mathcal{I}}^*(b) \neq \pi_{\mathcal{I}}^*(b') \implies
    a \land b \neq a'
    \land b'.
  \end{align*}

  For every $a \in \mathcal{A}$ define the horizontal section through $a$ by
  \begin{align*}
    \varphi(a) = \{b \in \mathcal{A}:\, (a,b) \in \mathcal{F}\}.
  \end{align*}
  Also, define the \emph{good} set
  \begin{align*}
    \mathcal{G} = \{a \in \mathcal{A}:\, |\varphi(a)| \geq 2^{-d-2}N\}.
  \end{align*}
  Let $M = |\mathcal{G}|$, then $M \geq 2^{-d-2}N$. Indeed,
  \begin{align*}
    \frac{N^2}{2^{d+1}} \leq |\mathcal{F}| &= \sum_{a \in \mathcal{A}}
    |\varphi(a)| \\
    &\leq \sum_{a \in \mathcal{G}} |\varphi(a)| +
    \sum_{a \notin \mathcal{G}} |\varphi(a)| \\
    &\leq MN + \frac{N^2}{2^{d+2}}.
  \end{align*}

  We now split the argument into three cases:
  \begin{itemize}
    \item $\mathcal{G}$ has only a few $\pi_{\mathcal{I}}$ projections.
    \item $\mathcal{G}$ has many different $\pi_{\mathcal{I}}$ projections but
      has sections with only a few $\pi_{\mathcal{I}}^*$ projections.
    \item $\mathcal{G}$ has many different $\pi_{\mathcal{I}}$ projections
      with sections having many different $\pi_{\mathcal{I}}^*$ projections.
  \end{itemize}
  Let us begin with the first case:
  \begin{align*}
    |\pi_{\mathcal{I}}(\mathcal{G})| \leq M^{\frac{k}{d-1}}.
  \end{align*}
  Then Proposition \ref{induction} applied to $\mathcal{G}$ with $C = 1$ gives
  \begin{align*}
    |\mathcal{G} \land \mathcal{G}| &\geq \meetconst(d-k)M^{\frac{d}{d-1}} \\
    &\geq 2^{-\gamma_{d-k}} M^{\frac{d}{d-1}}.
  \end{align*}
  Recall that $M \geq 2^{-d-2}N$, thus
  \begin{align*}
    |\mathcal{G} \land \mathcal{G}| \geq 2^{-\gamma_{d-k}}
    2^{-\frac{d(d+2)}{d-1}} N^{\frac{d}{d-1}} \geq 2^{-(\gamma_{d-k} + 2(d+2))}
    N^{\frac{d}{d-1}}.
  \end{align*}
  Observe that $\gamma$ satisfies the recurrence relation
  \begin{align} \label{gamma_rec}
    \gamma_{j} = 3(j+2) + \gamma_{j-1},
  \end{align}
  thus
  \begin{align*}
    \gamma_{d-k} + 2(d+2) \leq \gamma_{d-1} + 3(d+2) = \gamma_d
  \end{align*}
  and, since $\mathcal{G} \subseteq \mathcal{A}$, we arrive at
  $|\mathcal{A} \land \mathcal{A}| \geq |\mathcal{G} \land \mathcal{G}| \geq
  2^{-\gamma_d} N^{\frac{d}{d-1}}$.

  For the second and third cases, we will assume
  \begin{align*}
    |\pi_{\mathcal{I}}(\mathcal{G})| > M^{\frac{k}{d-1}}.
  \end{align*}
  Choose a subcollection $\mathcal{H} \subseteq \mathcal{G}$ on which
  $\pi_{\mathcal{I}}$ is injective; in particular,
  \begin{align*}
    |\mathcal{H}| \geq M^{\frac{k}{d-1}}.
  \end{align*}
  Since $\pi_{\mathcal{I}}$ is injective on $\mathcal{H}$, distinct
  $a \in \mathcal{H}$ have distinct $\pi_{\mathcal{I}}(a)$, so by the implication
  above the sets $a \land \varphi(a)$, $a \in \mathcal{H}$, are pairwise
  disjoint.

  Furthermore, we can assume $k \geq 2$. Otherwise, if $k=1$, then the
  projection $\pi_{\mathcal{I}}^*$ is injective (Lemma \ref{injective}),
  and we can argue as follows
  \begin{align*}
    |\mathcal{A} \land \mathcal{A}| &\geq \sum_{a \in \mathcal{H}} |a \land
    \varphi(a)| = \sum_{a \in \mathcal{H}} |\pi_{\mathcal{I}}^*(\varphi(a))|
    = \sum_{a \in \mathcal{H}} |\varphi(a)| \geq |\mathcal{H}|N2^{-d-2} \\
    &\geq M^{\frac{1}{d-1}} N 2^{-d-2} \geq N^{\frac{d}{d-1}} 2^{-2(d+2)}.
  \end{align*}
  Note that $\gamma_d \geq 2d + 4$ for $d \geq 3$, so
  \begin{align*}
    |\mathcal{A} \land \mathcal{A}| \geq 2^{-\gamma_d} N^{\frac{d}{d-1}}
  \end{align*}
  which is the claim. Therefore, we can continue the proof assuming
  $2 \leq k \leq d-2$. This makes
  \begin{align*}
    \mathcal{J} := [d] \setminus \mathcal{I}
  \end{align*}
  a proper index set with $|\mathcal{J}| = d-k$, where $2 \leq d-k \leq d-2$.

  Suppose now that there exists an $a \in \mathcal{H}$ with
  \begin{align*}
    |\pi_{\mathcal{J}}(\varphi(a))| \leq N^{\frac{d-k}{d-1}}.
  \end{align*}
  Set $\mathcal{E} = \varphi(a)$ and note that, since $\mathcal{H} \subseteq
  \mathcal{G}$, the definition of $\mathcal{G}$ gives $|\mathcal{E}| \geq
  2^{-d-2}N$. In particular
  \begin{align*}
    |\pi_{\mathcal{J}}(\mathcal{E})| \leq
    2^{\frac{(d+2)(d-k)}{d-1}} |\mathcal{E}|^{\frac{d-k}{d-1}}.
  \end{align*}
  Apply Proposition \ref{induction} and we obtain
  \begin{align*}
    |\mathcal{E} \land \mathcal{E}| &\geq \meetconst(k) C^{\frac{-1}{k-1}}
    |\mathcal{E}|^{\frac{d}{d-1}} \\
    &\geq \meetconst(k) C^{\frac{-1}{k-1}} 2^{-2(d+2)} N^{\frac{d}{d-1}},
  \end{align*}
  where $C = 2^{\frac{(d+2)(d-k)}{d-1}}$. Observe that
  \begin{align*}
    \frac{(d+2)(d-k)}{(d-1)(k-1)} \leq \frac{(d+2)(d-2)}{(d-1)(k-1)} \leq
    \frac{d+2}{k-1} \leq d+2,
  \end{align*}
  so
  \begin{align*}
    |\mathcal{E} \land \mathcal{E}| &\geq \meetconst(k) 2^{-3(d+2)}
    N^{\frac{d}{d-1}} \\
    &\geq 2^{-(\gamma_k + 3(d+2))} N^{\frac{d}{d-1}} \\
    &\geq 2^{-(\gamma_{d-1} + 3(d+2))} N^{\frac{d}{d-1}} \\
    &= 2^{-\gamma_{d}} N^{\frac{d}{d-1}},
  \end{align*}
  where we have used the recurrence relation \eqref{gamma_rec} in the last
  step. Since $\mathcal{E} = \varphi(a) \subseteq \mathcal{A}$, this gives
  $|\mathcal{A} \land \mathcal{A}| \geq |\mathcal{E} \land \mathcal{E}| \geq
  2^{-\gamma_d} N^{\frac{d}{d-1}}$.

  Finally, we assume that for every $a \in \mathcal{H}$ we have
  \begin{align*}
    |\pi_{\mathcal{J}}(\varphi(a))| \geq N^{\frac{d-k}{d-1}}.
  \end{align*}
  Then
  \begin{align*}
    |\mathcal{A} \land \mathcal{A}| &\geq \sum_{a \in \mathcal{H}} |a \land
    \varphi(a)| \\
    &\geq \sum_{a \in \mathcal{H}} |\pi_{\mathcal{J}}(\varphi(a))| \\
    &\geq |\mathcal{H}| N^{\frac{d-k}{d-1}}.
  \end{align*}
  Recall that $|\mathcal{H}| \geq M^{\frac{k}{d-1}} \geq 2^{-\frac{(d+2)k}{d-1}}
  N^{\frac{k}{d-1}} \geq 2^{-(d+2)} N^{\frac{k}{d-1}}$. Using once more that
  $\gamma_d \geq d+2$, we arrive at
  \begin{align*}
    |\mathcal{A} \land \mathcal{A}| \geq 2^{-\gamma_d} N^{\frac{d}{d-1}}.
  \end{align*}
  This is the last remaining case, which completes the proof.
\end{proof}
Theorem \ref{TheoremA} follows directly from Theorem \ref{theorem_constants},
with the explicit constant $c_d = 2^{-\gamma_d}$.

\begin{remark*}
  We have made no attempt to optimize the constant in
  Theorem \ref{theorem_constants}, and the exponent $\gamma_d = \Theta(d^2)$ is
  surely far from the truth. On the one hand,
  \begin{align*}
    \meetconst(d) \geq 2^{-\gamma_d},
  \end{align*}
  which decays like $2^{-\Theta(d^2)}$. On the other hand, taking $L \to \infty$
  in the simplex example shows
  \begin{align*}
    \meetconst(d) \leq \frac{\big((d-1)!\big)^{\frac{1}{d-1}}}{d}
    = \big(1 + o(1)\big)\frac{1}{e},
  \end{align*}
  a bounded quantity. The gap between these two bounds is enormous, and
  essentially all of the loss in our argument comes from the type pigeonhole of
  Proposition \ref{type_pigeonhole}, which keeps only a $2^{-d}$ fraction of the
  pairs and is invoked once at each of the $\sim d$ levels of the induction.
  We do not know the answer even to the following qualitative question.
\end{remark*}

\begin{question*}
  Is $\inf_{d \geq 2} \meetconst(d) > 0$? Equivalently, is there an absolute
  constant $c > 0$ such that every finite antichain
  $\mathcal{A} \subseteq \mathbb{Z}^d$ satisfies
  \begin{align*}
    |\mathcal{A} \land \mathcal{A}| \geq c\,|\mathcal{A}|^{\frac{d}{d-1}}
  \end{align*}
  in every dimension $d \geq 2$?
\end{question*}

\bibliography{antichains}
\bibliographystyle{abbrv}

\end{document}